\documentclass[12pt, reqno]{amsart}
\usepackage[bookmarksnumbered,colorlinks, plainpages]{hyperref}
\hypersetup{colorlinks=true,linkcolor=red, anchorcolor=green, citecolor=cyan, urlcolor=red, filecolor=magenta, pdftoolbar=true}
\usepackage{fixltx2e}

\RequirePackage{tkz-berge}
\usepackage{amsmath, amsthm, amscd, amsfonts, amssymb, graphicx, color, enumerate, xcolor,  mathrsfs, latexsym}

 \definecolor{cupgreen}{rgb}{0,0.498,0.208}
  \definecolor{cupblue}{rgb}{0,0,.5}
  \definecolor{cupred}{rgb}{1,0.04,0}
  \definecolor{cuppink}{rgb}{0.925,0,0.545}
  \definecolor{cupmagenta}{rgb}{0.624,0.161,0.424}
  \definecolor{cupbrown}{rgb}{0.71,0.212,0.133}

  \definecolor{cupgreen}{rgb}{0,0,0}
  \definecolor{cupblue}{rgb}{0,0,0}
  \definecolor{cupred}{rgb}{0,0,0}
  \definecolor{cuppink}{rgb}{0,0,0}
  \definecolor{cupmagenta}{rgb}{0,0,0}
  \definecolor{cupbrown}{rgb}{0,0,0}

\definecolor{TITLE}{rgb}{0,0,0}
\definecolor{AUTHOR1}{rgb}{0.00,0.59,0.00}
\definecolor{AUTHOR2}{rgb}{0.50,0.00,1.00}
\definecolor{SECTION}{rgb}{0.50,0.00,1.00}
\definecolor{FOOTTITLE}{rgb}{0.00,0.50,0.75}
\definecolor{THM}{rgb}{0.8,0,0.1}
\definecolor{SEC}{rgb}{0,0,1}

\makeatletter
\normalsize
\newtheorem{theorem}{{\color{THM} Theorem}}[section]
\newtheorem{lemma}[theorem]{{\color{THM}Lemma}}

\theoremstyle{definition}
\newtheorem{Conjecture}[theorem]{Conjecture}

\newtheorem{pro}[theorem]{{\color{THM}Problem}}

\numberwithin{equation}{section}

\usepackage{amscd}
\begin{document}
\title{Some results on average of Fibonacci and Lucas sequences}
%\author{Be\'{a}ta B\'{e}nyi}
%\address{Faculty of Water Sciences, National University of Public Service, Hungary}
%\email{beata.benyi@gmail.com}
\author{Daniel Yaqubi}
%\address{Faculty of Agriculture and Animal Science, University of Torbat-e Jam, Iran .}
\email{daniel\_yaqubi@yahoo.es}
\address{Department of Pure Mathematics, Ferdowsi University of Mashhad, Iran.}
\author{ Amirali Fatehizadeh}
\email{ amirali.fatehizadeh@gmail.com}
\address{Department of Mathematics, Quchan University of Technology, Iran.}
%\address{Faculty of Agriculture and Animal Science, University of Torbat-e Jam, Iran .}
\begin{abstract}
The numerical sequence in which the $n$-th term is the average (that is, arithmetic mean) of the of the first $n$ Fibonacci numbers may be found in the OEIS (see  \href{https://oeis.org/A111035 }{A111035 }). An interesting question one might pose is
which terms of the sequences
\begin{eqnarray}
\bigg(\frac{1}{n}\sum_{i=1}^n F_i\bigg)_{n=1}^{\infty} \text{and}~~\quad \bigg(\frac{1}{n}\sum_{i=1}^n L_i\bigg)_{n=1}^{\infty}
\end{eqnarray}
are integers?\\
 The average of the first three Fibonacci sequence, $(1+1+3)/3 = 4/3$, is not whereas the average of the first $24$  Fibonacci sequence, $\frac{\sum_{i=1}^{24}F_i}{24}=5058$ is. In this paper, we address this question and also present some properties of average Fibonacci and Lucas numbers using the \textit{Wall-Sun-Sun prime conjecture}.
\end{abstract}

\maketitle

\section{Introduction}
Fibonacci numbers originally arose in a problem in \textit{Liber Abaci}, published in 1202, which was one of the first texts to describe the Hindu-Arabic numeral system. Since then Fibonacci numbers have become one of the most popular sequences to study, appearing in a wealth of problems not only in enumerative combinatorics. \\
The Fibonacci numbers, denoted by $\{F_n\}_{n=0}^{\infty}$ , are defined by the following recurrence relation
\begin{equation}
F_{n+1}=F_n+F_{n-1},
\end{equation}
with initial values $F_0=0$ and $F_1=1$ . The first elements of this sequence are given in (\href{https://oeis.org/A000045 }{A000045 }), as
\[0\quad 1\quad1\quad 2\quad 3\quad 5\quad 8\quad 13\quad 21\quad \ldots .\]
The Fibonacci numbers are closely related to binomial coefficients; it is a well-known fact that they are given by the sums of the rising diagonal lines of Pascal's triangle (see \cite{Kos})
\[F_{n+1}=\sum_{i=0}^{\lfloor\frac{n}{2}\rfloor} {n-i \choose i}.\]
The Lucas numbers $\{L_n\}_{n=0}^{\infty}$, are defined by the same recurrence relation as the Fibonacci numbers with different initial values (see \href{https://oeis.org/A000032}{A000032 }). In a similar, way the Lucas sequence satisfies the recurrence relation
\begin{equation}
L_{n+1}=L_n+L_{n-1},\quad L_0=2, \quad L_1=1.
\end{equation}
Fibonacci and Lucas numbers have been extensively studied. A closed form of the \emph{Binet's formula}, for instance, can be expressed in terms of the characteristic roots of the
Fibonacci sequence.
Let $\alpha$ and $\beta$ denote the roots of the polynomial $x^2-x-1=0$, i.e., $\alpha=\frac{1+\sqrt{5}}{2}$ and $\beta=\frac{1-\sqrt{5}}{2}$. 
\begin{align*}
F_n=\frac{1}{\sqrt{5}}({\alpha}^n-{\beta}^n)\quad \mbox{and}\quad
L_n={\alpha}^n+{\beta}^n.
\end{align*}
These formulas can be extended to negative integers $n$ in a natural way with $F_{-n} =(-1)^{n-1}F_n$
and $L_{-n}=(-1)^nL_n$ for all $n$.

In this paper, we pose the following question with a slightly
different formulation.
\begin{pro}
which terms of the sequences
\begin{eqnarray}
\bigg(\frac{1}{n}\sum_{i=1}^n F_i\bigg)_{n=1}^{\infty} \text{and}~~\quad \bigg(\frac{1}{n}\sum_{i=1}^n L_i\bigg)_{n=1}^{\infty}
\end{eqnarray}
are integers?
\end{pro}
%We investigate the sequence of the averages of the first $n$  Fibonacci numbers for $n=1,2,\ldots$.
 The average of the Fibonacci and Lucas sequences is not always an integer, for instance, Fibonacci sequence for  $n=3$ have $\frac{1+1+2}{3}=\frac{4}{3}$. However, for $n=24$ for instance, we obtain an integer as the average. We will be interested in examining for which $n$ is this average is an integer.

The paper is organized as follows. First, in Section 2, we recall some well-known facts about the Fibonacci and Lucas numbers that will be used later. Section 3 includes the main theorem and finally, in Section 4, we gave further properties on the average of Fibonacci and Lucas numbers using the \textit{Wall-Sun-Sun prime conjecture}.
%================
\section{Preliminaries}
First, we give some identities and theorems about Fibonacci and Lucas numbers which we will make use of later\cite{Kos}.\\
\begin{equation}\label{Ide1}
2|F_n\Leftrightarrow 2|L_n\Leftrightarrow 3|n;
\end{equation}
\begin{equation}\label{Ide2}
F_1+F_2+F_3+\cdots F_n=F_{n+2}-1; 
\end{equation}
\begin{equation}\label{Luc1}
L_1+L_2+L_3+\cdots L_n=L_{n+2}-3; 
\end{equation}
\begin{equation}\label{Ide3}
F_{2k}=F_{k+1}^2-F_{k-1}^2;
\end{equation}
\begin{equation}\label{Ide4}
24|F_{12k};
\end{equation}
\begin{equation}\label{Ide5}
\text{for}~n\geq3~~F_n|F_m\Leftrightarrow n|m.
\end{equation}

We also recall the following known formulas for Fibonacci and Lucas sequences \cite{Bri}.
\begin{theorem}\label{Rec}
For any positive integer $k$, we have
\begin{eqnarray*}
F_{4k+2}-1&=&F_{2k}L_{2k+2}\\
F_{4k+3}-1&=&F_{2k+2}L_{2k+1}\\
F_{4k+4}-1&=&F_{2k+3}L_{2k+1}\\
F_{4k+5}-1&=&F_{2k+2}L_{2k+3}.
\end{eqnarray*}
\end{theorem} 

\begin{lemma}\label{Lem1}
Let $n$ be a positive integer. Then
\[L_n\equiv L_{n+6}\pmod{4}.\]
\end{lemma}
\begin{proof}
We know that the Lucas sequence satisfies the recurrence relation  
\begin{equation}\label{equ2}
L_{n}=L_{n-1}+L_{n-2},
\end{equation}
with the initial values $L_0=2$ and $L_1=1$. Consider the first values of this sequence. It is easy to see $L_1\equiv L_{7}\pmod{4},\quad L_2\equiv L_{8}\pmod{4},\quad L_3\equiv L_{9}\pmod{4}$, and so on. The claim follows by induction on
$n$ using the formula \eqref{equ2}.
\end{proof}
\begin{lemma}\label{Lem2}
Let $k$ be a be a positive integer. The sum of three consecutive Lucas numbers with odd index is divisible by $4$, i.e.,
\[L_{2k+1}+L_{2k+3}+L_{2k+5}\equiv 0 \pmod{4}.\]
\end{lemma}
\begin{proof}
Consider the first values of the Lucas sequence that have odd index. The
proof follows by induction on $k$ as in the proof of Lemma \eqref{Lem1}.
\end{proof}

\begin{lemma}\label{Lem3}
Let $k$ be a positive integer. Then
\[L_{2k+2}\not\equiv 0 \pmod{4}.\]
%$x\nmid y 3\not|
\end{lemma}
\begin{proof}
We show that $L_{2k+2}\equiv 2 \pmod{4}$ whenever the Lucas number $L_{2k+2}$ is even. Using the well-known formulas for the Lucas sequence, we can write
\[L_{2k+2}-2=L_{2(k+1)}-2=\sum_{i=1}^{k+1}L_{2i-1}.\]
According \eqref{Ide1}, since $L_{2k+2}$ is an even number when $3|k+1$. Hence, for positive integers $k'$, we have
\[ \sum_{i=1}^{k+1}L_{2i-1}=\sum_{i=1}^{3k'}L_{2i-1}.\]
We will show by induction on $k'$ that
\[\sum_{i=1}^{3k'}L_{2i-1}\equiv 0\pmod{4}.\]
 For $k' =1$ we have 
\[\sum_{i=1}^{3}L_{2i-1}=16\equiv 0\pmod{4}.\]
Assume now that it holds for all positive integers less than some $k'\geq 2$. For $k'+1$ we have
\[\sum_{i=1}^{3(k'+1)}L_{2i-1}=\sum_{i=1}^{3k'}L_{2i-1}+\sum_{i=3k'+1}^{3k'+3}L_{2i-1}.\]
Using the Lemma \eqref{Lem2}, $\sum_{i=3k'+1}^{3k'+3}L_{2i-1}\equiv 0\pmod{4}$, which 
completes the proof.
\end{proof}
%========
\section{Main Theorems}
In this section, we focus on the average of Fibonacci sequences (OEIS: \href{https://oeis.org/A111035 }{A111035 }). Looking at the first few values of this sequence we see that most of these numbers are divisible by $24$. \\
In our proof, we use often identity \eqref{Ide2}, which states that the sum of the first $n$ Fibonacci numbers is $F_{n+2}-1$. (this sequence also appears in the OEIS, see \href{https://oeis.org/A000071 }{A000071 }).
\begin{theorem}\label{Thm2}
Let $n$ be a positive integer. There are infinitely many numbers such that
\[n \vert\sum_{i=1}^n F_i.\]
\end{theorem}
\begin{proof}
For a positive integer $k$, consider $n=4k$ such that $4k|\sum_{i=1}^{4k} F_i$. Identity \eqref{Ide2} and Theorem  \eqref{Rec} give
\[4k|\sum_{i=1}^{4k} F_i=F_{4k+2}-1=F_{2k}L_{2k+2}.\]
By Lemma \eqref{Lem3}, $4k\not\vert L_{2k+2}$, which implies that $4k|F_{2k}$. Identity \eqref{Ide3} thus yields
\[4|F_{k+1}^2-F_{k-1}^2.\]
This means that $F_{k+1}$ and $F_{k-1}$ are either both even, or they are both odd.
It follows that $F_k$ is even, thus according identity \eqref{Ide1} we may set $k=3k'$ and so $12k'|F_{6k'}$.
 Put $k'=2k''$, which gives $24k''|F_{12k''}$. According  \eqref{Ide4}, the following holds
\begin{equation}\label{Not}
k''|F_{12k''}
\end{equation} 
Now, to conclude the proof, it is sufficient to show  that there are infinitely many numbers such that $k''|F_{12k''}$. Put $k''=F_{3\ell}$, so $F_{3\ell}|F_{12F_{3\ell}}$, by \eqref{Ide5}, $3\ell|12F_{3\ell}$. Since $3|12F_{3\ell}$, we have $\ell|F_{3\ell}$.

Continuing this process we put $\ell=F_{3d}$, with $d$ a positive integer. So, $F_{3d}|F_{3F_{3d}}$, and this leads to $3d|3F_{3d}$, thus,  $d|F_{3d}$. Continuing this process infinitely many times, we have that $k''$ is equal to the following sequence 
\[k''=F_{3F_{3F_{3F_{\ddots}}}}.\]
\end{proof}
In the proof of the Theorem \eqref{Thm2}, using the
properties of the Fibonacci numbers, there are infinitely many numbers $k''$ such that $k''|F_{k''}$ and $F_{k''}|F_{12k''}$, We have, there are infinity many numbers such that $k''|F_{12k''}$.\\
In the following theorem, we proved the result of the Theorem \eqref{Thm2} is true for Lucas sequence. First, we recall the following known formula for Lucas sequences (see \cite{Kos}, page $111$). 
\begin{lemma}\label{Luca2}
Let $m\geq n$ be positive integers. Then
$$L_{m+n}-L_{m-n}=
\left\{
   \begin{array}{ll}
       
                L_mL_n                 & \text{if $ n$ is odd}\\
                5F_mF_n& \quad       \text{otherwise.}
         \end{array}
         \right.
         $$
\end{lemma}
\begin{theorem}\label{Luca3}
Let $n$ be a positive integer. There are infinitely many numbers such that
\[n \vert\sum_{i=1}^n L_i.\]
\end{theorem}
\begin{proof}
According Lemma \eqref{Luca2}, for $n=2k$ and $m=2k+2$ we have
\[L_{m+n}-L_{m-n}=L_{4k+2}-L_2=L_{4k+2}-3=5F_{2k+2}F_{2k}.\]
 Identity \eqref{Luc1} give
 \[\sum_{i=1}^{4k} L_i=F_{4k+2}-3=5F_{2k}F_{2k+2}.\]
To conclude the proof, it is sufficient consider
 $n=4k$ such that $4k|\sum_{i=1}^{4k} L_i=5F_{2k}F_{2k+2}$. 
Now, with similar proof of the Theorem \eqref{Thm2},  there are infinitely many numbers $k$, such that $4k|F_{2k}$.
\end{proof}
\begin{theorem}\label{Thm3}
Let $\alpha$ be a non-negative integer. Then
\begin{equation}
3.2^{\alpha+3} | \sum_{i=1}^{3.2^{\alpha+3}}F_i=F_{3.2^{\alpha+3}+2}-1.
\end{equation}
\end{theorem}
\begin{proof}
Let $k\geq 1$ be a positive integer. We will use the Fibonacci identities 
\begin{equation*}
F_{4k+2}-1=F_{2k}L_{2k+2}~\text{and}~ F_{2k}=F_kL_k.
\end{equation*}
Therefore
\begin{eqnarray*}
F_{3.2^{\alpha+3}+2}-1&=&F_{3.2^{\alpha+2}}L_{3.2^{\alpha+2}+2}\\
&=&F_{3.2^{\alpha+1}}L_{3.2^{\alpha+1}}L_{3.2^{\alpha+2}+2}\\
&=&F_{3.2^{\alpha}}L_{3.2^{\alpha}}L_{3.2^{\alpha+1}}L_{3.2^{\alpha+2}+2}\\
&=&\cdots\\
&=&F_3L_3L_6\cdots L_{3.2^{\alpha}}L_{3.2^{\alpha+1}}L_{3.2^{\alpha+2}+2}.
\end{eqnarray*}
But $F_3=2,~L_3=4$ and each of $L_6,\cdots, L_{3.2^{\alpha+1}}$ are divisible by $2$. Also, $L_{3.2^{\alpha+2}+2}$ is divisible  by $3$. This concludes the proof.
\end{proof}
The following theorem is a generalization of  Theorem \eqref{Thm3}.
\begin{theorem}\label{Thmn}
Let $\alpha,\beta$ and $\gamma$ be positive integers. Then
\begin{equation}
2^{\alpha+3}.3^{\beta+1}.5^{\gamma} \vert \sum_{i=1}^{2^{\alpha+3}.3^{\beta+1}.5^{\gamma}} F_i.\
\end{equation}
\end{theorem}
\begin{proof}
In Theorem \eqref{Thm3}, we proved the cases $\alpha=\beta=\gamma=0$. Now, let $\beta,\gamma >0$. Then
\[4|2^{\alpha+3}.3^{\beta+1}.5^{\gamma} .\]
By the proof of Theorem \eqref{Thm2}, there are positive integers $k$  such that 
\[24k|F_{12k}.\]
By \eqref{Ide4}, it can be seen that, 
\begin{equation}\label{Par}
k|F_{12k}
\end{equation} 
To avoid any confusion, we set
$k=2^{\alpha}.3^{\beta}.5^{\gamma}.$
 Now, by utilizing \eqref{Par}, we get 
 \[2^{\alpha}.3^{\beta}.5^{\gamma}|F_{2^{\alpha+2}.3^{\beta+1}.5^{\gamma}}.\]
 To conclude the proof we need the following three steps:
 \begin{itemize}
\item[] \textbf{Step I.} 
 First, we shall prove that $2^{\alpha}|F_{3.2^{\alpha+2}}$. The proof of this
step can be achieved by induction. For $\alpha=1$, the claim is true, since $2|F_{3.8}=46368 $. Suppose the claim is true for all positive integers less than $\alpha$. We show that the same is true for $\alpha + 1$, that is $2^{\alpha+1}|2F_{3.2^{\alpha+3}}.$ By multiplying to $2$, we can write
\[2^{\alpha+1}|2F_{3.2^{\alpha+2}}.\]
Note that
\[2F_{3.2^{\alpha+2}}|F_{3.2^{\alpha+3}}.\]
Since $2|F_3=2$ and $F_3|F_{3.2^{\alpha+3}}$, we have by way of identity \eqref{Ide5} that
\[F_{3.2^{\alpha+2}}|F_{3.2^{\alpha+3}}.\]
Hence, $2^{\alpha+1}|2F_{3.2^{\alpha+3}}.$
\item[] \textbf{Step II.}  Next we show that $3^{\beta}|F_{2^{\alpha+2}.3^{\beta+1}.5^{\gamma}}$ by induction on $\beta$. For $\beta=1$, according \eqref{Ide5},  $F_3| F_{36.2^{\alpha}.5^{\gamma}}$. Assume the statement holds for all positive integers less than $\beta$. We need to show that the statement also holds for $\beta+1$. Multiplying both sides of the induction hypothesis by 3 gives
\[3^{\beta+1}|3 F_{2^{\alpha+2}.3^{\beta+1}.5^{\gamma}}.\]
To complete the proof of this step, it is sufficient to show
\[3 F_{2^{\alpha+2}.3^{\beta+1}.5^{\gamma}}| F_{2^{\alpha+2}.3^{\beta+2}.5^{\gamma}}.\]
Much like in the previous step, we have $3|F_{2^{\alpha+2}.3^{\beta+2}.5^{\gamma}}$  
and
\[F_{2^{\alpha+2}.3^{\beta+1}.5^{\gamma}}|F_{2^{\alpha+2}.3^{\beta+2}.5^{\gamma}},\]

 whence 
 \[3^{\beta+1}|F_{2^{\alpha+2}.3^{\beta+2}.5^{\gamma}},\]
 which concludes step II.
%Similarly as in the previous step, since $3|F_{2^{\alpha+2}.3^{\beta+2}.5^{\gamma}}$ , also
%\[F_{2^{\alpha+2}.3^{\beta+1}.5^{\gamma}}|F_{2^{\alpha+2}.3^{\beta+2}.5^{\gamma}},\] which implies 
%\[3^{\beta+1}|F_{2^{\alpha+2}.3^{\beta+2}.5^{\gamma}}.\]
%the statement of the Step II.
\item[] \textbf{Step III.}
Finally, using the same arguments of the previous steps, with induction on $\gamma$, we show $5^{\gamma}|F_{5^{\gamma}}$. For $\gamma=1$ it is easy to see that $5|F_5=5$. By induction,  we assume that the statement is true for all positive integers less than $\gamma$. We have to prove that the claim is true for $\gamma+1$. By multiplying the induction hypothesis by $5$,  we obtain
\[5^{\gamma+1}|5F_{5^{\gamma}}.\]
Now, we have to show \[5F_{5^{\gamma}}|F_{5^{\gamma+1}}.\]
Again, since $5|F_{5^{\gamma+1}}$ and $F_{5^{\gamma}}|F_{5^{\gamma+1}}$.
\[5^{\gamma+1}|F_{5^{\gamma+1}}.\]
\end{itemize}
%By Step I, we have $2^{\alpha}|F_{2^{\alpha+2}.3^{\beta+1}.5^{\gamma}}$. Step II led  $3^{\beta}|F_{2^{\alpha+2}.3^{\beta+1}.5^{\gamma}}$ and Step III led $ 5^{\gamma}|F_{2^{\alpha+2}.3^{\beta+1}.5^{\gamma}}$. 
%
According the results of the previous steps, we have
\[
2^{\alpha}.3^{\beta}.5^{\gamma}|F_{2^{\alpha+2}.3^{\beta+1}.5^{\gamma}}.\]
Therefore,

\[2^{\alpha+3}.3^{\beta+1}.5^{\gamma}|\sum_{i=1}^{2^{\alpha+2}.3^{\beta+1}.5^{\gamma}}F_i.\]
\end{proof}
%===============================
In the same manner of the Theorem \eqref{Thmn}, we obtain the following Theorem.
\begin{theorem}
Let $\alpha,\beta$ and $\gamma$ be positive integers. Then
\begin{equation}
2^{\alpha+3}.3^{\beta+1}.5^{\gamma} \vert \sum_{i=1}^{2^{\alpha+3}.3^{\beta+1}.5^{\gamma}} L_i.\
\end{equation}
\end{theorem}
%\begin{proof}
% Let $\beta,\gamma >0$. Then
%\[4|2^{\alpha+3}.3^{\beta+1}.5^{\gamma} .\]
%By the proof of Theorem \eqref{Thm2} and \eqref{Luca3}, there are positive integers $k$  such that 
%\[24k|F_{12k}.\]
%By \eqref{Ide4}, it can be seen that, 
%\begin{equation}\label{Par}
%k|F_{12k}
%\end{equation} 
%Put $k=2^{\alpha}.3^{\beta}.5^{\gamma}.$
% Now, the proof follows
%by similar way  with the proof of Theorem \eqref{Thmn}.
%\end{proof}
\section{ Prime numbers and the average of Fibonacci numbers }
In this section, we discuss Theorem \eqref{Thm2} for some special $n$. First, we consider the sum of the first $n$ Fibonacci numbers, when $n$ is a prime number, $p$.  
Before we present our results, we first recall some well-known definitions and theorems. 

 An integer $a$ is called a \textit{quadratic residue modulo $p$} (with $p>2$) if $p\not|a$ and there exists an integer $b$ such that  $a\equiv b^2 \pmod{p}$. Otherwise, it is called a \emph{non-quadratic residue modulo $p$}.\\
\textit{Legendre} introduced the following practical notation:
\begin{eqnarray}
\left(\frac{a}{b}\right)=
\begin{cases}
+1 \hspace{1.5cm} \text{if $a$ is a quadratic residue modulo $p$} ; \\
0 \hspace{1.8cm} \text{if $p$ divides $a$} ;  \\
-1\hspace{1.5cm} \text{otherwise}.
\end{cases}
\end{eqnarray}
We note that for a prime number $p$, and $b=5$,  it is easy to see that the Legendre symbol, $\Big(\frac{p}{5}\Big)$, is equal to
\begin{eqnarray}
\ \left(\frac{p}{5}\right)=
\begin{cases}
+1 \hspace{1.5cm} \text{if $p\equiv \pm 1\pmod{5}$ } ; \\
0 \hspace{1.8cm} \text{if $p\equiv 0\pmod{5}$} ;  \\
-1\hspace{1.5cm} \text{if $p\equiv \pm 2\pmod{5}$}.
\end{cases}
\end{eqnarray}
We shall denote by $\sigma(n)$ the rank of apparition of the Lucas sequence ${L_n}$, if it exists. According \cite{Mor}, the \textit{rank of apparition} of the sequence $\{S_n\}$ is the smallest index $k$ such that $m|S_k$ for some non-zero element $S_k$, provided it exists. The rank of apparition of the Fibonacci sequence $\{F_n\}$ is denoted by $\rho(n)$. 
These numbers are sometimes called \textit{Fibonacci periods} or \textit{Pisano periods}. The initial values of $\rho(n)$ is given OEIS, \href{https://oeis.org/A001602 }{A001602 }.\\
The following lemma concerning Fibonacci and Lucas numbers is well-known (see \cite{Rib}, page $41-55$).
\begin{lemma}\label{Lem4}
Let $n$ be a positive integer and  $\rho(n)$ be the rank of apparition of the Fibonacci sequence. Then
\begin{itemize}
\item [i)] {$m|F_n$ if and only if $\rho(m)|n$;}
\item [ii)] {If $p|F_n$, then $p^e|F_{np^{e-1}}$ for $e\geq 1$;}
\item [iii)] {$(F_n,L_n)|2$.}
\end{itemize}
\end{lemma}
The property (ii) of the Lemma \eqref{Lem4} known as the \textit{law of
apparition } of Lucas sequences of the first kind, in general. We will also need the following lemma from \cite{Mor}.
\begin{lemma}\label{Lem5}
The odd prime power $p^r$ is a divisor of the Lucas sequence $\{L_n\}$ if and only if $\rho(p^r)$ is even. If $p^r$ is a divisor of the sequence $L_n$, then
\[\sigma(p^r)=\frac{\rho(p^r)}{2},\] and
\[p^r|L_n\Longleftrightarrow n\equiv \frac{\rho(p^r)}{2}\pmod{\rho(p^r)}.\]
\end{lemma}

%In Chapter 5, Section III, I shall discuss this question for the special
%case when b = 1. I will show how it is ultimately related, in a very
%surprising way,It is also an interesting problem to estimate the size of the largest
%prime factor of an - bn, where a > b ~ 1. The following notation
%will be used: P(m) = largest prime factor of m.
%It is not difficult to show, using

We are now ready to state our theorem. 
\begin{theorem}\label{10}
Let $p$ be an odd prime number. Then
\begin{equation}
p\not\vert \sum_{i=1}^p F_i.
\end{equation}
\end{theorem}
\begin{proof}
Assume that $p$ is an odd prime number such that
\[p\vert \sum_{i=1}^p F_i.\]
We investigate the cases $p=4k+1$ and $p=4k+3$ separately.
 \begin{itemize}
\item[] \textbf{Case I.} 
Suppose that $p=4k+1$. According \eqref{Ide2} and Theorem \eqref{Rec}, we have
\begin{equation}
p\vert \sum_{i=1}^p F_i=F_{4k+3}-1=F_{2k+2}L_{2k+1}.
\end{equation}
Hence, $p|F_{2k+2}$ or $p|L_{2k+1}$. Suppose first that $p|F_{2k+2}$. According Lemmas \eqref{Lem4} and \eqref{Lem5}, there exists a positive integer $k'$ such that
\[k'\left(p-\left(\frac{p}{5}\right)\right)=2k+2.\]
This implies
\[k'\left(4k+1-\left(\frac{p}{5}\right)\right)=2k+2,\]
which is impossible, since $\left(\frac{p}{5}\right)=\pm 1$ .

Now, suppose that  $p|L_{2k+1}$. From the identity $F_{2n}=L_nF_n$ we have $p|F_{4k+2}$. So, by Lemma \eqref{Lem5}, there exists a positive integer $k''$ such that 
\begin{equation}
4k+2=k''\left(\frac{p-\left(\frac{p}{5}\right)}{2}\right),
\end{equation}
which  implies
\[8k+4=k''\left(4k+1-\left(\frac{p}{5}\right)\right).\]
We have again a contradiction, since  $\left(\frac{p}{5}\right)=\pm 1$.
\item[] \textbf{Case II.} 
Assume now that  $p=4k+3$. We follow the argument given in
Case I. According \eqref{Ide2} and Theorem \eqref{Rec}, we can write
\begin{equation}
p\vert \sum_{i=1}^p F_i=F_{4(k+1)+1}-1=F_{2(k+1)}L_{2(k+1)+1}.
\end{equation}
The prime number $p=4k+3$ must divide $F_{2(k+1)}$ or $L_{2(k+1)+1}$. Suppose first that $p|F_{2(k+1)}$.  By Lemma \eqref{Lem4}, there exists a positive integer $k'$ such that
\[k'\left(p-\left(\frac{p}{5}\right)\right)=2(k+1).\]
Hence,
\[k'\left(4k+3-\left(\frac{p}{5}\right)\right)=2k+2.\]
which cannot hold, since $\left(\frac{p}{5}\right)=\pm 1$.

Now, let $p|L_{2(k+1)+1}$. According Lemma \eqref{Lem5}, there exists a positive integer $k''$ such that 
\begin{equation}
4(k+1)+2=k''\left(\frac{p-\left(\frac{p}{5}\right)}{2}\right),
\end{equation}
which implies
\[k''\left(4k+3-\left(\frac{p}{5}\right)\right)=8(k+1)+4.\]
This is impossible, since  $\left(\frac{p}{5}\right)=\pm 1$ .
\end{itemize}
Proof of Theorem\eqref{Thm10} follows immediately from the above two contradictions.
%++++++++++++++++++++++++++++++++++
%
\end{proof}
Consider now the first values of the sequence in OEIS \href{https://oeis.org/A111035 }{A111035 }:
\begin{eqnarray*}
&&\textcolor{red}{1}, 2, 24, 48, 72,\textcolor{red}{ 77}, 96, 120, 144, 192, 216, 240, 288, \textcolor{red}{319}, \\&&\textcolor{red}{323},
 336, 360, 384, 432, 480, 576, 600, 648, 672, 720, 768,\\&& 864, 960, 
 1008, 1080, 1104, 1152, 1200, 1224, 1296, 1320,\\&& 1344, 1368, 1440, 
\textcolor{red}{ 1517}, 1536, 1680, 1728, 1800, 1920\cdots . 
\end{eqnarray*}
We see that each odd integer within this sequence is square-free, for example, $77=11\times 7\quad, 319=11\times29,\quad323=17\times 19$ and $1517=37\times41.$ By assuming the validity of the well-known conjecture concerning Wall-Sun-Sun (or Fibonacci-Wieferich) primes, we show in Theorem \eqref{Thm9} below that if $n$ is an odd positive integer such that $n|\sum_{i=1}^nF_i$ then $n$ is square-free.\\
We first state the Wall-Sun-Sun Prime Conjecture:
\begin{Conjecture}[Wall-Sun-Sun Prime Conjecture \cite{Sun}]\label{Con}
There are no prime numbers $p$ such that
\[p^2|F_{p-\left(\frac{p}{5}\right)}.\]
\end{Conjecture}
In $1992$, the authors of \cite{Sun} proved that if \textit{Wall-Sun-Sun prime number conjecture} is true, then the \textit{Fermat equation} $x^p+y^p=z^p$ has no integral solutions with $p\not\vert xyz$. Empirically it has been observed that there are no Wall-Sun-Sun primes less than $100,000,000,000,000.$

\begin{theorem}\label{Thm9}
Let $n$ be an odd positive integer. If $n|\sum_{i=1}^n F_i$, then $n$ is a square-free.
\end{theorem}

%For a positive integer $n$, the \textit{Pisano period or Fibonacci periods $k(n)$} is the minimal period of the Fibonacci sequence to modulo $n$. These numbers are often called the \textit{Wall number}  because there were discovered by D. D. Wall \cite{Wal}. The sequence of $k(n)$ is given in OEIS \href{https://oeis.org/A001175 }{A001175 }.
%
%For proof of the Conjecture \eqref{Con}, we need to the following theorem.

\begin{proof}
Let $n=p_1^{\alpha_1}p_2^{\alpha_2}\cdots p_k^{\alpha_k}$, where each $p_i$ is an odd prime and suppose $n|\sum_{i=1}^nF_i$. We will show that $\alpha_1=\alpha_2=\cdots=\alpha_k=1$.\\
 By Theorem \eqref{Rec} and identity \eqref{Ide2}, for positive integers $a,b$, we have
 \begin{equation}
 n\vert \sum_{i=1}^n F_i=F_aL_b
 \end{equation}
 so for $1\leq i\leq k$ we have $p_i^{\alpha_i}\vert F_aL_b$.\\ 
 It follows from the above argument that $p_i^{\alpha_i}|F_a$ or $p_i^{\alpha_i}|L_b$ so suppose $p_i^{\alpha_i}\vert  F_a$. Employing Lemma \eqref{Lem4} and  the Wall-Sun-Sun-prime conjecture \eqref{Con}, we have
 \[p_i^2\not\vert  F_{p_i-\left(\frac{p_i}{5}\right)},\]
 which implies $p_i^2\not\vert F_a$.
 
 Suppose now that $p_i^{\alpha}|L_b$. Again, according Lemma \eqref{Lem5} and the Wall-Sun-Sun Prime Conjecture we have
 \[p_i^2\not\vert L_b.\]
 It follows that $p_i^2\not\vert F_aL_b$ for $1\leq i\leq k$, which implies that $\alpha_1=\alpha_2=\cdots=\alpha_k=1$, so we see immediately that $n$ is square-free.
\end{proof}
We conclude this paper with interesting conjectures concerning averages of Fibonacci numbers.
\begin{Conjecture}\label{Con4}
For each positive integer $t$, there is positive integer like $n$ such that if  
\[n\vert \sum_{i=1}^n F_i, \quad \text{then} \quad n+t\vert \sum_{i=1}^{n+t} F_i.\]
\end{Conjecture}
\begin{Conjecture}\label{Con3}
There are infinitely many pairs of positive integers $(n,n+1)$ such that 
\[n\vert \sum_{i=1}^n F_i\quad \text{and}\quad n+1\vert \sum_{i=1}^{n+1} F_i.\]
\end{Conjecture}
Using Mathematica we have empirically found all such pairs of integers up to 1000000. These pairs are:
\[(1, 2), (6479, 6480), (11663, 11664), (51983,51984), (196559, 196560).\]

\end{document}